\documentclass[11pt]{article}

\usepackage{amsfonts,amsmath,latexsym,color,epsfig,hyperref}
\newtheorem{theorem}{Theorem}[section]
\newtheorem{lemma}[theorem]{Lemma}
\newtheorem{proposition}[theorem]{Proposition}

\newtheorem{corollary}[theorem]{Corollary}

\newenvironment{proof}
      {\medskip\noindent{\bf Proof:}\hspace{1mm}}
      {\hfill$\Box$\medskip}
\def\qed{\ifvmode\mbox{ }\else\unskip\fi\hskip 1em plus 10fill$\Box$}

\input{epsf}

\makeatletter
\def\Ddots{\mathinner{\mkern1mu\raise\p@
\vbox{\kern7\p@\hbox{.}}\mkern2mu
\raise4\p@\hbox{.}\mkern2mu\raise7\p@\hbox{.}\mkern1mu}}
\makeatother

\title{\vspace{-0.7cm}Short proofs of some extremal results II}
\author{David Conlon\thanks{Mathematical Institute, Oxford OX2 6GG,
United Kingdom. Email: {\tt david.conlon@maths.ox.ac.uk}. Research
supported by a Royal Society University Research Fellowship.}\and
Jacob Fox\thanks{Department of Mathematics, Stanford University, Stanford, CA 94305. Email: {\tt fox@math.mit.edu}. Research supported by a Packard Fellowship, by NSF Career Award DMS-1352121 and by an Alfred P. Sloan Fellowship.}
\and
Benny Sudakov\thanks{Department of Mathematics, ETH, 8092 Zurich, Switzerland.
Email: {\tt benjamin.sudakov@math.ethz.ch}. Research supported by SNSF grant 200021-149111.}}
\date{}

\begin{document}
\maketitle

\begin{abstract}
We prove several results from different areas of extremal combinatorics, including complete or partial solutions to a number of open problems. These results, coming mainly from extremal graph theory and Ramsey theory, have been collected together because in each case the relevant proofs are quite short. 
\end{abstract}

\section{Introduction}

We study several questions from extremal combinatorics, a broad area of discrete mathematics which deals with the problem of maximizing or minimizing the cardinality of a collection of finite objects satisfying a certain property. The problems we consider come mainly from the areas of extremal graph theory and Ramsey theory. In many cases, we give complete or partial solutions to open problems posed by researchers in the area.

While each of the results in this paper is interesting in its own right, the proofs are all quite short. Accordingly, in the spirit of Alon's `Problems and results in extremal 
combinatorics' papers~\cite{Al1, Al2, Al3} and our own earlier paper~\cite{CFS14}, we have chosen to combine them. We describe the results in brief below. For full details on each topic we refer the reader to the relevant 
section, each of which is self-contained and may be read separately from all others.

In Section~\ref{sec:setmap}, we address old questions of Erd\H{o}s and Hajnal~\cite{EH58} and Caro~\cite{C87} concerning extremal problems for set mappings. In Section~\ref{sec:FKP}, we answer a question of Foucaud, Krivelevich and Perarnau~\cite{FKP}, which extends an old problem of Bollob\'as and Erd\H{o}s~\cite{E68}, about finding large $K_{r,r}$-free subgraphs in graphs with a given number of edges. In Section~\ref{sec:cube}, we show how to use the Lov\'asz local lemma to embed sparse hypergraphs in large dense hypergraphs and apply this technique to improve the Ramsey number of the cube and other bipartite graphs. In Section~\ref{sec:minors}, we address a problem of Erd\H{o}s and Hajnal \cite{EH89} on estimating the bounds on a variant of the classical Ramsey problem and extend this to estimate the extremal number for small and shallow clique minors. We find a connection between induced Ramsey numbers and the Ruzsa--Szemer\'edi induced matching problem in Section~\ref{sec:matchings}. Finally, in Section~\ref{sec:coloredremoval}, we prove a colored variant of the famous triangle removal lemma with reasonable bounds.

All logarithms are base $2$ unless otherwise stated. For the sake of clarity of presentation, we systematically omit floor and ceiling signs whenever they are not crucial. We also do not make any serious attempt to optimize absolute constants in our statements and proofs. 

\section{Extremal problems for set mappings} \label{setmappings} \label{sec:setmap}

Consider all mappings $f:{M \choose k} \to {M \choose l}$ with $|M|=m$ such that $X$ is disjoint from $f(X)$ for all $X \in {M \choose k}$.
Let $p(m,k,l)$ be the maximum $p$ such that for every such mapping $f$ there is a subset $P \subset M$ with $|P| = p$ where $f(X)$ and $P$ are disjoint for all $X \in {P \choose k}$. In 1958, Erd\H{o}s and Hajnal \cite{EH58} proved that, for all $k$ and $l$, $$cm^{1/(k+1)} \leq p(m,k,l) \leq c'(m \log m)^{1/k},$$
where $c$ and $c'$ depend only on $k$ and $l$, and asked for a more exact determination of the dependence of $p(m,k,l)$ on $m$. 

In 1972, in a textbook application of the probabilistic method, Spencer \cite{S72} proved an extension of Tur\'an's theorem to $k$-uniform hypergraphs and used this to prove that $p(m,k,l) \geq cm^{1/k}$, where $c$ depends only on $k$ and $l$. More precisely, Spencer showed that every $r$-uniform hypergraph with $n$ vertices and $t$ edges contains an independent set of order at least $c_rn^{1+1/(r-1)}/t^{1/(r-1)}$. Now consider the $(k+1)$-uniform hypergraph with $m$ vertices and the $t=l{m \choose k}$ edges given by $X \cup \{y\}$ for all $X \in {[m] \choose k}$ and $y \in f(X)$. An independent set $P$ in this hypergraph is a set for which $f(X)$ and $P$ are disjoint for all $X \in {P \choose k}$. Since there is an independent set of order at least $\Omega(m^{1+1/k}/t^{1/k})=\Omega(m^{1/k})$, we get the desired lower bound on $p(m,k,l)$. 

Despite the attention the set mapping problem has received over the years, the upper bound has not been improved. Here we solve the Erd\H{o}s--Hajnal problem when $l$ is sufficiently large as a function of $k$, showing that $p(m, k, (k-1)!) = \Theta(m^{1/k})$, where the implied constants depend only on $k$. For simplicity, we first describe a construction for $l=k!$ and then show how to modify it to get
$l=(k-1)!$.

\begin{theorem}
For $m= n^k$, there is a function $f:{M \choose k} \to {M \choose k!}$ such that $|M|=m$, $X$ and $f(X)$ are disjoint for all $X \in {M \choose k}$, and every $P \subset M$ of order greater than $k^2 n$ contains a $k$-set $X$ such that $f(X)$ is not disjoint from $P$. 
\end{theorem}  
\begin{proof}
Let $M=[n]^k$ and consider the function $f:{M \choose k} \to {M \choose k!}$ such that if $X$ consists of the $k$-tuples 
$(x_{1i},x_{2i},\ldots,x_{ki})$ for $i \in [k]$, then $f(X)$ consists of all $k$-tuples $(x_{1\pi(1)},x_{2\pi(2)},\ldots,x_{k\pi(k)})$, where $\pi$ is a permutation of $[k]$, with the caveat that if $(x_{1\pi(1)},x_{2\pi(2)},\ldots,x_{k\pi(k)})$ is equal to an element of $X$ or a previously chosen element of $f(X)$, we instead choose an arbitrary element of $M$ which is not equal to any of these elements. With this choice, $f(X)$ is well defined and $X$ and $f(X)$ are disjoint.

Let $P \subset [n]^k$ with $|P|>k^2 n$. As long as one of the hyperplanes with one fixed coordinate has at most $k$ elements of $P$, delete those elements from $P$. As there are at most $kn$ such hyperplanes, there are at most $k^2 n$ deleted points. Therefore, there is a remaining point $p \in P$ such that there are at least $k$ other points remaining on each of the $k$ hyperplanes containing $p$ with one coordinate fixed. Picking distinct points $p_1,\ldots,p_k \in P$ which are distinct from $p$ with $p_i$ having the same $i^{\textrm{th}}$ coordinate as $p$ and letting $X=\{p_1,\ldots,p_k\}$, we have $p \in f(X)$, which completes the proof.  
\end{proof}

To improve the bound on $l$ in this theorem from $k!$ to $(k-1)!$, we need a minor modification. Indeed, suppose that the elements of $M$ are ordered lexicographically (that is, order is determined by first comparing the first coordinates, then the second coordinates and so on) and $f(X)$ is that subset of the $k$-tuples $(x_{1\pi(1)},x_{2\pi(2)},\ldots,x_{k\pi(k)})$ for which $\pi(1)$ is fixed so that the first coordinate $x_{1\pi(1)}$ is equal to the minimum of the first coordinates of the elements of $X$. The proof then proceeds as above, but we choose $p$ and $p_1$ to be the smallest, in the lexicographic ordering, among the remaining vertices after all deletions. This guarantees that $p_1$ is the first element in the set $\{p_1, \dots, p_k\}$, though the remaining elements may be ordered arbitrarily. We note that $l=(k-1)!=1$ is best possible when $k = 2$. After this paper was written, we learned that this particular case was independently solved much earlier by F\"uredi~\cite{F91}. It would be very interesting to determine whether $l$ can be decreased to $1$ for all $k$.

A related question of Caro~\cite{C87} (see also~\cite{AC86}) asks for an estimate on $q(m, k, d)$, the maximum $q$ such that for every mapping $f:{M \choose k} \to {M \choose k}$ with $|M|=m$ such that $|X \cap f(X)| \leq d$ for all $X \in {M \choose k}$, there is a subset $Q \subset M$ with $|Q| = q$ such that $f(X)$ is not a subset of $Q$ for any $X \in {Q \choose k}$. This differs from the Erd\H{o}s--Hajnal question on three counts: we take $l = k$; we allow $X$ to overlap with $f(X)$ by a certain controlled amount $d$; and we only require a subset where the image of each element is not contained within the subset rather than being entirely disjoint from it. If we let $t = (k-d)/(2k - d - 1)$, Caro~\cite{C87} proved that 
\[c m^t \leq q(m, k, d) \leq c' (m \log m)^t,\]
where $c$ and $c'$ depend only on $k$ and $d$. Here we partially answer a question of Caro~\cite{C87} by 
removing the log factor from the upper bound when $k = 2$.

\begin{theorem}
There exist constants $c_1$ and $c_2$ such that
\begin{itemize}
\item[(i)]
$q(m, 2, 1) \leq c_1 m^{1/2}$,
\item[(ii)]
$q(m, 2, 0) \leq c_2 m^{2/3}$.
\end{itemize}
\end{theorem}

\begin{proof}
(i) We define a mapping of the complete graph with vertex set $[m]^2$. If $x < x'$ and $y \neq y'$, we map the edge $((x,y),(x',y'))$ to $((x,y),(x,y'))$ and otherwise we map arbitrarily while ensuring that $((x,y), (x',y'))$ doesn't map to itself. 


Suppose now that $Q$ is a subset of order at least $2m + 1$. On every horizontal or vertical line, we delete the highest point which is in $Q$. Since we delete at most $2m$ points, some point $q \in Q$ must remain. If $q = (x, y')$, we see that there are points $(x, y)$ and $(x',y')$ with $x < x'$ and $y' < y$ which are also in $Q$. But, if $e = ((x,y), (x', y'))$, its image is $((x,y), (x, y'))$, which is also in $Q$. 

(ii) We define a mapping of the complete graph with vertex set $[m]^3$. If $x < x'$, $y \neq y'$ and $z \neq z'$, we map the edge $((x,y,z), (x', y', z'))$ to $((x',y,z),(x',y,z'))$ and otherwise we map arbitrarily while ensuring that $((x,y,z)), (x', y', z'))$ is disjoint from its image.

Suppose now that $Q$ is a subset of order at least $3m^2 + 1$. On every line $\{(a,y,z): 1 \leq a \leq m\}$, we remove the lowest point, while on every line $\{(x,b,z): 1 \leq b \leq m\}$, we remove the highest point. The remaining set still has $m^2 + 1$ points. Therefore, there are two points which have the same $x$ and $y$ coordinate, say, $(x', y, z)$ and $(x', y, z')$, where $z < z'$. Since we removed the highest point on the line $\{(x', b, z'): 1 \leq b \leq m\}$, there exists $y' > y$ such that $(x', y', z')$ is in $Q$. Similarly, since we removed the lowest point on the line $\{(a, y, z): 1\leq a \leq m\}$, there exists $x < x'$ such that $(x, y, z)$ is in $Q$. But, if $e = ((x,y,z), (x', y', z'))$, its image is $((x',y,z),(x',y,z'))$, which is also in $Q$.
\end{proof}

\section{Large subgraphs without complete bipartite graphs} \label{sec:FKP}

Given a family of graphs $\mathcal{H}$, we let $f(m, \mathcal{H})$ be the size of the largest $\mathcal{H}$-free subgraph that can be found in any graph with $m$ edges, where a graph is $\mathcal{H}$-free if it contains no graph from the family $\mathcal{H}$ as a subgraph. The problem of estimating $f(m, \mathcal{H})$ was first raised by Bollob\'as and Erd\H{o}s~\cite{E68} at a workshop in 1966, where they asked whether every graph with $m$ edges contains a $C_4$-free subgraph with $\Omega(m^{3/4})$ edges. Erd\H{o}s~\cite{E71} later remarked that the answer in this case is likely $\Theta(m^{2/3})$, based on an example due to Folkman and private communication from Szemer\'edi.

Recently, the function $f(m, \mathcal{H})$ was rediscovered by Foucaud, Krivelevich and Perarnau~\cite{FKP}, who considered the case where $\mathcal{H}$ is the family of even cycles of length at most $2k$, obtaining estimates that are tight up to a logarithmic factor. In this section, we address a question asked by these same authors and extend the Folkman--Szemer\'edi result by determining $f(m, H)$ up to a constant factor when $H$ is a complete bipartite graph. 

Let $K_{r,s}$ be the complete bipartite graph with parts of order $r$ and $s$, where $2 \leq r \leq s$. The following theorem gives a lower bound on $f(m,K_{r,s})$.

\begin{theorem} \label{th1}
Every graph $G$ with $m$ edges contains a $K_{r,r}$-free subgraph of size at least $\frac{1}{4}m^{\frac{r}{r+1}} $.
\end{theorem}

To prove this theorem, we need an upper bound on the number of copies of $K_{r,r}$ which can be found in a graph with $m$ edges.
The problem of maximizing the number of copies of a fixed graph $H$ over all graphs with a given number of edges was solved by Alon \cite{Al} (and the corresponding problem for hypergraphs was solved by
Friedgut and Kahn \cite{FK}). For our purposes, the following simpler estimate will suffice.

\begin{lemma} \label{l1}
Every graph $G$ with $m$ edges contains at most $2m^r$ copies of $K_{r,r}$.
\end{lemma}

\begin{proof}
Note that every copy of $K_{r,r}$ in $G$ contains a matching of size $r$. Since the number of such matchings is at most 
${m \choose r}$ and every matching of size $r$ can appear in at most $2^r$ copies of $K_{r,r}$, the total number of such copies is at most $2^r{m \choose r} \leq 2m^r$.
\end{proof}

Using this lemma, together with a simple probabilistic argument, one can prove the required lower bound on $f(m,K_{r,s})$.
\vspace{3mm}

\noindent{\bf Proof of Theorem \ref{th1}:}\, Let $G$ be a graph with $m$ edges. Consider a subgraph $G'$ of $G$ obtained by choosing
every edge independently at random with probability $p=\frac{1}{2}m^{-1/(r+1)}$.
Then the expected number of edges in $G'$ is $mp$. Also, by
Lemma \ref{l1}, the expected number of copies of $K_{r,r}$ in $G'$ is at most $2p^{r^2}m^{r}$. Delete one edge from every copy of 
$K_{r,r}$ contained in $G'$. This gives a $K_{r,r}$-free subgraph of $G$, which, by linearity of expectation, has at least
$$pm-2p^{r^2}m^{r} \geq \frac{1}{2} m^{\frac{r}{r+1}}-\frac{1}{8} m^{\frac{r}{r+1}} \geq \frac{1}{4} m^{\frac{r}{r+1}}$$
edges on average. Hence, there exists a choice of $G'$ which produces a $K_{r,r}$-free subgraph of $G$ of size at least
$\frac{1}{4}m^{\frac{r}{r+1}} $. \hfill $\Box$
\vspace{0.1cm}

We will now show that this estimate is tight when $G$ is an appropriately chosen complete bipartite graph with $m$ edges. Since the Tur\'an number for $K_{r,s}$ is not known in general, it is somewhat surprising that one can prove a tight bound on the size of the largest $K_{r,s}$-free subgraph in graphs with $m$ edges.

\begin{theorem} \label{th2}
Let $2 \leq r \leq s$ and let $G$ be a complete bipartite graph with parts $U$ and $V$, where $|U|=m^{1/(r+1)}$ and $|V|=m^{r/(r+1)}$. Then $G$ has $m$ edges and the largest $K_{r,s}$-free subgraph of $G$ has at most $s m^{r/(r+1)}$ edges.
\end{theorem}

\begin{proof} 
The proof is a simple application of the counting argument of  K\H{o}v\'ari, S\'os and Tur\'an \cite{KST}. Let $G'$ be a $K_{r,s}$-free subgraph of $G$ and let $d=e(G')/|V|$ be the average degree of vertices in $V$ within $G'$. If $d \geq s$, then, by convexity,
$$\sum_{v \in V} { d_{G'}(v) \choose r} \geq |V| {d \choose r} \geq {s \choose r}m^{r/(r+1)}\geq s m^{r/(r+1)}/r!\,.$$
On the other hand, since $G'$ is $K_{r,s}$-free we have that
$$\sum_{v \in V} {d_{G'}(v) \choose r} < s {|U| \choose r} \leq s |U|^r/r!=s m^{r/(r+1)}/r!\,.$$
This contradiction completes the proof of the theorem. 
\end{proof}


These results can also be extended to $k$-uniform hypergraphs, which, for brevity, we call $k$-graphs. Let $K^{(k)}_{r,\ldots,r}$ denote the complete $k$-partite $k$-graph with parts of order $r$.

 \begin{theorem}  \label{th3}
 Every $k$-graph $G$ with $m$ edges contains a $K^{(k)}_{r,\ldots,r}$-free subgraph of size at least $\frac{1}{2k!}m^{\frac{q-1}{q}} $, where
 $q=\frac{r^k-1}{r-1}$.
 \end{theorem}

\begin{proof}
Let $G$ be a $k$-graph with $m$ edges. Every copy of $K^{(k)}_{r,\ldots,r}$ in $G$ contains a matching of size $r$ and the number of such matchings is at most ${m \choose r}$. On the other hand, every matching in $G$ of size $r$ can appear in at most $(k!)^r$ copies of $K^{(k)}_{r,\ldots,r}$. This implies that the total number of such copies is at most $(k!)^r{m \choose r}$. 

Consider a subgraph $G'$ of $G$ obtained by choosing
every edge independently at random with probability $p=\frac{1}{k!}m^{-1/q}$.
Then the expected number of edges in $G'$ is $mp$ and the expected number of copies of $K^{(k)}_{r,\ldots,r}$ in $G'$ is at most 
$(k!)^rp^{r^k}{m \choose r}$. Delete one edge from every copy of 
$K^{(k)}_{r,\ldots,r}$ contained in $G'$. This gives a $K^{(k)}_{r,\ldots,r}$-free subgraph of $G$ with at least
$$pm- (k!)^rp^{r^k}{m \choose r} \geq \frac{1}{2k!}m^{\frac{q-1}{q}} $$
expected edges. Hence, there exists a choice of $G'$ which produces a $K^{(k)}_{r,\ldots,r}$-free subgraph of $G$ of this size.
\end{proof}

We can again show that this estimate is tight when $G$ is an appropriately chosen $k$-partite $k$-graph.

\begin{theorem} \label{th4}
Let $k, r \geq 2$,  $q=\frac{r^k-1}{r-1}$ and let $G$ be a complete $k$-partite $k$-graph with parts $U_i, 1\leq i \leq k$, such that $|U_i|=m^{r^{i-1}/q}$. Then $G$ has $m$ edges and the largest $K^{(k)}_{r,\ldots,r}$-free subgraph of $G$ has $O(m^{(q-1)/q})$ edges.
\end{theorem}

This result follows from the next statement, which is proved using a somewhat involved extension of the counting argument used in the graph case. This technique has its origins in a paper of Erd\H{o}s~\cite{E64}. Throughout the proof, we use the notation $\binom{t}{r}$ as a shorthand for $\frac{t(t-1)\dots(t-r+1)}{r!} 1_{t \geq r-1}$, thus extending the definition of the binomial coefficient to a convex function on all of $\mathbb{R}$.

\begin{proposition}
Let $G$ be a $k$-partite $k$-graph with parts $U_i, 1\leq i \leq k$, such that $|U_i|=n^{r^{i-1}}$ and with $a \prod_{i \geq 2}|U_i|$ edges.
Then $G$ contains at least ${a - k + 1\choose r}\prod_{i \leq k-1} {|U_i| \choose r}$ copies of $K^{(k)}_{r,\ldots,r}$.
\end{proposition}

\begin{proof} 
We prove the result by induction on $k$. We may always assume that $a > r + k - 2$, since otherwise $\binom{a - k + 1}{r} = 0$ and the result is trivial. First, suppose that $k=2$ and we have a bipartite graph with parts $U_1$ of order $n$,
$U_2$ of order $n^r$ and $an^r$ edges. Let $d(S)$ denote the number of common neighbors of a subset $S$ in $G$ and let 
$D=\sum_{S\subset U_1, |S|=r}d(S)/{n \choose r}$ be the average number of common neighbors taken over all subsets of order $r$ in $U_1$. 
Note that  
$$\sum_{S\subset U_1, |S|=r}d(S)=\sum_{x \in U_2} { d(x) \choose r} \geq {a \choose r} |U_2|={a \choose r}n^r.$$
Therefore, since $a > r$, $D \geq {a \choose r}n^r/{n \choose r} \geq a$ and the number of copies of $K_{r,r}$ in $G$ is
$$ \sum_{S\subset U_1, |S|=r}{d(S) \choose r} \geq {D \choose r}   {n \choose r} \geq {a \choose r} {n \choose r},$$
completing the proof in this case.

Now suppose we know the statement for $k-1$. For every vertex $x \in U_k$, let $G_x$ be the $(k-1)$-partite $(k-1)$-graph which is the link of vertex $x$ (i.e., the collection of all subsets of order $k-1$ which together with $x$ form an edge of $G$). Let 
$a_x \prod_{i=2}^{k-1}|U_i|$ be the number of edges in $G_x$. By definition, $\sum_x a_x=a|U_k|=an^{r^{k-1}}$. By the induction hypothesis, each 
$G_x$ contains at least ${a_x - k + 2 \choose r}\prod_{i \leq k-2} {|U_i| \choose r}$ copies of 
$K^{(k-1)}_{r,\ldots,r}$. By convexity, the total number of such copies added over all $G_x$ is at least
\begin{align*}
{a - k + 2 \choose r}n^{r^{k-1}}\prod_{i \leq k-2} {|U_i| \choose r} & = {a - k + 2 \choose r}|U_{k-1}|^r\prod_{i \leq k-2} {|U_i| \choose r}\\
& \geq r!{a - k + 2 \choose r} \prod_{i \leq k-1} {|U_i| \choose r} \geq (a - k + 2)\prod_{i \leq k-1} {|U_i| \choose r},
\end{align*}
where in the final inequality we use that $a > r + k - 2$.

For every subset $S$ which intersects every $U_i$, $i\leq k-1$, in exactly $r$ vertices, let $d(S)$ be the number of vertices $x \in U_k$ such that
$x$ forms an edge of $G$ together with every subset of $S$ of order $k-1$ which contains one vertex from every $U_i$. By the above discussion, we have that 
$$\sum_S d(S) \geq  (a - k + 2)\prod_{i \leq k-1} {|U_i| \choose r},$$ 
that is, at least the number of copies of $K^{(k-1)}_{r,\ldots,r}$ counted over all $G_x$.
On the other hand, by the definition of $d(S)$,  the number of copies of $K^{(k)}_{r,\ldots,r}$ in $G$ equals
$\sum_S {d(S) \choose r}$. Since the total number of sets $S$ is $\prod_{i \leq k-1} {|U_i| \choose r}$, the average value of $d(S)$ is at least $a-k+2$ and the result now follows by convexity.
\end{proof}

\section{Ramsey numbers and embedding large sparse hypergraphs into dense hypergraphs} \label{sec:cube}

For a graph $H$, the {\it Ramsey number} $r(H)$ is the least positive integer $N$ such that every two-coloring of the edges of the complete graph $K_N$ contains a monochromatic copy of $H$. One of the most important results in graph Ramsey theory is a theorem of Chv\'atal, R\"odl, Szemer\'edi and Trotter~\cite{CRST83} which says that for every positive integer $\Delta$ there is a constant $C(\Delta)$ such that every graph $H$ with $n$ vertices and maximum degree $\Delta$ satisfies $r(H) \leq C(\Delta) n$. That is, the Ramsey number of bounded-degree graphs grows linearly in the number of vertices.

The original proof of this theorem used the regularity lemma and gives a very poor tower-type bound for $C(\Delta)$. Following improvements by Eaton~\cite{E98} and Graham, R\"odl and Ruci\'nski~\cite{GRR00}, the bound $C(\Delta) \leq 2^{c \Delta \log \Delta}$ was given by the authors~\cite{CFS12}. This is close to optimal, since Graham, R\"odl and Ruci\'nski~\cite{GRR00, GRR01} showed that there exist graphs $H$ (even bipartite graphs) with $n$ vertices and maximum degree $\Delta$ for which $r(H) \geq 2^{c' \Delta} n$. 
 
If we assume that $H$ is bipartite, work of Conlon~\cite{C09} and Fox and Sudakov~\cite{FS09} shows that $r(H) \leq 2^{c \Delta} n$ for any $H$ with $n$ vertices and maximum degree $\Delta$. By the results of Graham, R\"odl and Ruci\'nski mentioned above, this is optimal up to the constant $c$. The bound proved by Fox and Sudakov~\cite{FS09}, $r(H) \leq \Delta 2^{\Delta + 5} n$, remains the best known. Here we remove the $\Delta$ factor from this bound.

\begin{theorem} \label{thm:bip}
For every bipartite graph $H$ on $n$ vertices with maximum degree $\Delta$, $r(H) \leq 2^{\Delta + 6}n$.
\end{theorem}

This theorem allows us to give a slight improvement on the Ramsey number of cubes. The {\it $d$-cube} $Q_d$ is the $d$-regular graph on $2^d$ vertices whose vertex set is $\{0,1\}^d$ and where two vertices are adjacent if they differ in exactly one coordinate. Burr and Erd\H{o}s~\cite{BE75} conjectured that $r(Q_d)$ is linear in the number of vertices $|Q_d|$. After several improvements over the trivial bound $r(Q_d) \leq r(|Q_d|) \leq 4^{|Q_d|} = 2^{2^{d+1}}$ by Beck~\cite{B83}, Graham, R\"odl and Ruci\'nski~\cite{GRR01} and Shi~\cite{S01,S07}, Fox and Sudakov~\cite{FS09} obtained the bound $r(Q_d) \leq d 2^{2d + 5}$, which is nearly quadratic in the number of vertices. This follows immediately from their general upper bound on the Ramsey numbers of bipartite graphs. Theorem~\ref{thm:bip} improves this to a true quadratic bound.

\begin{corollary} \label{cor:cube}
For every positive integer $d$, $r(Q_d) \leq 2^{2d + 6}$.
\end{corollary}

We note that results similar to Theorem~\ref{thm:bip} and Corollary~\ref{cor:cube} were proved by Lee~\cite{L15} in his recent breakthrough work on the Ramsey numbers of degenerate graphs. However, the method he uses is very different to ours. To understand our approach, it will be useful to first describe the method used in~\cite{FS09} to prove the bound $r(H) \leq \Delta 2^{\Delta+5}n$.

Suppose that $H$ is a bipartite graph with $n$ vertices and parts $V_1$ and $V_2$, where every vertex in $V_1$ has degree at most $\Delta$ and every vertex in $V_2$ has degree at most $k$. The proof from~\cite{FS09} has two main ingredients. The first ingredient is a powerful probabilistic technique known as dependent random choice (see, for example, the survey~\cite{FS11} for a discussion of its many variants and applications) which allows one to find a large vertex subset $U$ in a dense graph $G$ such that almost all subsets of at most $k$ vertices from $U$ have many common neighbors. 

To prove an upper bound on the Ramsey number of $H$, we take $G$ to be the denser of the two monochromatic graphs which edge-partition the complete graph $K_N$, so that $G$ has edge density at least $1/2$. We use the dependent random choice lemma to find a subset $U$ with the property that almost every subset with at most $k$ vertices has at least $n$ common neighbors. We then form an auxiliary hypergraph $\mathcal{G}$ on $U$ by letting a subset $S$ with at most $k$ vertices be an edge of $\mathcal{G}$ if the vertices of $S$ have at least $n$ common neighbors in $G$. We also define a hypergraph $\mathcal{H}$ on $V_1$ by saying that a subset $T$ with at most $k$ vertices is an edge if there is a vertex of $H$ (which will necessarily be in $V_2$) whose neighborhood is $T$. It is easy to show that if $\mathcal{H}$ is a subhypergraph of $\mathcal{G}$, then $H$ is a subgraph of $G$. Thus, to prove an upper bound on Ramsey numbers, it suffices to show that every sparse hypergraph is a subhypergraph of every not much larger but very dense hypergraph. An embedding lemma of this form is the second ingredient used in~\cite{FS09}.

To state the appropriate lemma from~\cite{FS09},  we say that a hypergraph is down-closed if $e_1 \subset e_2$ and $e_2 \in E$ implies $e_1 \in E$.

\begin{lemma}\cite{FS09} \label{lem:FSembed}
Let $\mathcal{H}$ be an $n$-vertex hypergraph with maximum degree $\Delta$ such that each edge of $\mathcal{H}$ has size at most $k$ and suppose that $\delta \leq (4\Delta)^{-k}$. If $\mathcal{G}$ is a down-closed hypergraph on vertex set $U$ with $N \geq 4n$ vertices and more than $\left(1-\delta\right){N \choose k}$  edges of size $k$, then there is a copy of $\mathcal{H}$ in $\mathcal{G}$.
\end{lemma}

The proof of this lemma uses a greedy embedding process. However, we may improve it by a simple application of the Lov\'asz local lemma, which we now recall.

\begin{lemma} \label{lem:LLL}
Let $A_1,\ldots,A_n$ be events in an arbitrary probability space. A directed graph $D=(V,E)$ on the set of vertices $V=\{1,\ldots,n\}$ is called a dependency digraph for the events $A_1,\ldots,A_n$ if for each $i$, $1 \leq i \leq n$, the event $A_i$ is mutually independent of all the events $\{A_j:(i,j) \not \in E\}$. Suppose $D=(V,E)$ is a dependency digraph for the above events and suppose there are real numbers $x_1,\ldots,x_n$ such that $0 \leq x_i < 1$ and $\textrm{Pr}[A_i] \leq x_i \prod_{(i,j) \in E} (1-x_j)$ for all $1 \leq i \leq n$. Then $$\textrm{Pr}\left[ \bigwedge_{i=1}^n \bar A_i \right] \geq \prod_{i=1}^n \left(1-x_i\right).$$
In particular, with positive probability no event $A_i$ holds.
\end{lemma}

Using this result, we now improve Lemma~\ref{lem:FSembed} as follows.

\begin{lemma} \label{lem:embed}
Let $\mathcal{H}$ be an $n$-vertex hypergraph with maximum degree $\Delta$ such that each edge of $\mathcal{H}$ has size at most $k$ and suppose that $\delta \leq \frac{1}{4k\Delta}2^{-8kn/N}$. If $\mathcal{G}$ is a down-closed hypergraph on vertex set $U$ with $N \geq 16n$ vertices and more than $\left(1-\delta\right){N \choose k}$  edges of size $k$, then there is a copy of $\mathcal{H}$ in $\mathcal{G}$.
\end{lemma}

\begin{proof}
When $k = 1$, the result follows from Lemma~\ref{lem:FSembed}. We may therefore assume that $k \geq 2$.

Consider a uniform random mapping $f:V(\mathcal{H}) \rightarrow V(\mathcal{G})$. For two vertices $uv$ of $\mathcal{H}$, consider the bad event $A_{uv}$ that $f(u)=f(v)$. For an edge $e$ of $\mathcal{H}$, we also consider the bad event $B_e$ that $|f(e)|=|e|$ (the vertices of $e$ map to distinct vertices), but $f(e)$ is not an edge of $\mathcal{G}$. 

Clearly, $\textrm{Pr}[A_{uv}]=1/N$. We also have $\textrm{Pr}[B_e] \leq\delta$. Indeed, suppose $|e|=\ell$. If $f(e)$ is not an edge in $\mathcal{G}$, then, since $\mathcal{G}$ is down-closed, none of the ${N-\ell \choose k-\ell}$ $k$-sets containing it are in $\mathcal{G}$ either. However, the number of pairs consisting of an edge of size $k$ which is not in $\mathcal{G}$ and a subset of size $\ell$ is at most $\delta \binom{N}{k} \binom{k}{\ell}$. It follows that the number of $\ell$-sets which are not edges of $\mathcal{G}$ is at most 
$$\delta {N \choose k} \cdot \frac{{k \choose \ell}}{{N-\ell \choose k-\ell}}=\delta{N \choose \ell},$$
which implies that $\textrm{Pr}[B_e] \leq\delta$.

Provided $\{u',v'\}$ and $e$ are disjoint from $\{u,v\}$, the event $A_{uv}$ is independent of $A_{u'v'}$ and $B_e$. Therefore, $A_{uv}$ is dependent on at most $2(n-2)$ of the events $A_{u'v'}$ and at most $2 \Delta$ of the events $B_e$. Similarly, provided $\{u,v\}$ and $e'$ are disjoint from $e$, the event $B_e$ is independent of $A_{uv}$ and $B_{e'}$. Therefore, $B_e$ is dependent on at most ${n \choose 2}-{n-|e| \choose 2} < kn$ of the events $A_{uv}$ and at most $k \Delta$ of the events $B_{e'}$.

We now apply the local lemma. For each $A_{uv}$, we let the corresponding $x_i$ be $x$ and, for each $B_e$, we let the corresponding $x_i$ be $y$. Let $x=\frac{4}{N}$ and $y=\frac{1}{2k\Delta}$. As $x,y \leq 1/2$, we have $1-x \geq 4^{-x}$ and $1-y \geq 4^{-y}$. Therefore, since $N \geq 16 n$ and $k \geq 2$,
$$x(1-x)^{2n}(1-y)^{2\Delta} \geq \frac{4}{N}4^{-8n/N}4^{-1/k} \geq \frac{1}{N}=\textrm{Pr}[A_{uv}]$$ 
and 
$$ y(1-y)^{k\Delta}(1-x)^{kn} \geq \frac{1}{2k\Delta}4^{-1/2}4^{-4kn/N} \geq \delta \geq \textrm{Pr}[B_e].$$
By Lemma~\ref{lem:LLL}, the probability that none of the bad events $A_{uv}$ and $B_e$ occur is positive and, therefore, $\mathcal{H}$ is a subhypergraph of $\mathcal{G}$. 
\end{proof}

To prove Theorem~\ref{thm:bip}, we need an appropriate variant of the dependent random choice lemma. The version we use follows easily from Lemma~2.1 of~\cite{FS09}.

\begin{lemma} \label{lem:DRC}
Let $G$ be a bipartite graph with parts $V_1$ and $V_2$ of order $N$ and at least $\epsilon N^2$ edges, where $N \geq \epsilon^{-k} \max(bn, 4k)$. Then there is a subset $U \subset V_1$ with $|U| \geq 2^{-1/k}\epsilon^k N$ such that the number of $k$-sets $S \subset U$ with $|N(S)|< n$ is less than $2^{k+1} b^{-k}{|U| \choose k}$. 
\end{lemma}

Combining Lemmas~\ref{lem:embed} and~\ref{lem:DRC}, we arrive at the following theorem. 

\begin{theorem}
Let $H$ be a bipartite graph with $n$ vertices such that one part has maximum degree $k$ and the other part has maximum degree $\Delta$. If $G$ is a bipartite graph with edge density $\epsilon$ and at least $16\Delta^{1/k}\epsilon^{-k}n$ vertices in each part, then $H$ is a subgraph of $G$. 
\end{theorem}

\begin{proof}
Let $N = 16\Delta^{1/k}\epsilon^{-k}n$. Applying Lemma~\ref{lem:DRC} with $b = 16 \Delta^{1/k} \geq 2(8 k \Delta)^{1/k} 2^{8n/N}$, we find a set $|U|$ with $|U| \geq 2^{-1/k} \epsilon^k N \geq 16n$ vertices such that the number of $k$-sets $S \subset U$ with $N(S) < n$ is less than $2^{k+1} b^{-k}{|U| \choose k}$. Since $2^{k+1} b^{-k} \leq \frac{1}{4k\Delta} 2^{-8kn/N}$, we may apply Lemma~\ref{lem:embed} to embed the auxiliary hypergraph $\mathcal{H}$ in the hypergraph $\mathcal{G}$ (as described before Lemma~\ref{lem:FSembed}). This in turn implies that $H$ is a subgraph of $G$.
\end{proof}

By considering the denser color in any two-coloring, this result has the following immediate corollary. Theorem~\ref{thm:bip} follows as a special case.

\begin{corollary}
If $H$ is a bipartite graph with $n$ vertices such that one part has maximum degree $k$ and the other part has maximum degree $\Delta$, then $r(H) \leq \Delta^{1/k}2^{k+5}n.$
\end{corollary}

\section{Weakly homogeneous sequences and small minors} \label{sec:minors}

In 1989, Erd\H{o}s and Hajnal~\cite{EH89} studied an extension of the fundamental problem of estimating Ramsey numbers. A sequence $S_1,\ldots,S_t$ of disjoint vertex subsets of a graph is called a {\it weakly complete $r$-sequence of order $t$} if each subset $S_i$ has cardinality $r$ and, for each pair $1 \leq i < j \leq t$, there is an edge from a vertex in $S_i$ to a vertex in $S_j$. Let $g(r,n)$ be the largest $t$ for which every graph on $n$ vertices or its complement contains a weakly complete $r$-sequence of order $t$. Note that determining $g(1,n)$ is simply the classical Ramsey problem, since the vertices making up $S_1,\ldots,S_t$ form either a clique or an independent set. 

For $r$ fixed and $n$ sufficiently large, Erd\H{o}s and Hajnal~\cite{EH89} proved that 
$$(1/2 - o_r(1))(3/2)^{r}\log n \leq g(r,n) \leq 2^{r^2+1}r\log n,$$ 
where the upper bound comes from considering a random $2$-coloring of the edges of the complete graph on $n$ vertices. 
These estimates naturally lead one to ask whether the power of $r$ in the exponent of the constant factor should be $1$ or $2$.
Improving the lower bound of Erd\H{o}s and Hajnal, we answer this question by showing that the upper bound is much closer to the truth. Moreover, for $r \geq 2$, we will show that a density theorem holds, that is, every dense graph contains a large weakly complete $r$-sequence. 

A sequence $S_1,\ldots,S_t,T_1,\ldots,T_t$ of disjoint vertex subsets of a graph is called a {\it weakly bi-complete $r$-sequence of order $t$} if $|S_i|=|T_i|=r$  for $1 \leq i \leq t$ and, for each pair $1 \leq i, j \leq t$, there is an edge from a vertex in $S_i$ to a vertex in $T_j$. Given such a weakly bi-complete $r$-sequence, $S_1 \cup T_1,\ldots,S_t \cup T_t$ is clearly a weakly complete $2r$-sequence. 

\begin{theorem}\label{thm1}
Let $n$ be sufficiently large and $G$ be a graph with $n$ vertices and edge density $p$. The graph $G$ contains a weakly bi-complete $r$-sequence of order $t$ if
\begin{itemize} 
\item 
$p \geq n^{-1/3}$, $r \leq 2p^{-1/2}$ and $t \leq \frac{\log n}{4\log (32/pr^2)}$;
\item  
$p \geq n^{-1/5}$, $4p^{-1/2} \leq r \leq \sqrt{p^{-1}\log n}$ and $t \leq \frac{1}{16}e^{pr^2/8}\log n$; 
\item  
$r \geq 4\sqrt{p^{-1}\log n}$ and $t \leq \min(pn/64\sqrt{\log n}, n/2r)$. 
\end{itemize}
In particular, $G$ also contains a weakly complete $2r$-sequence of order $t$. 
\end{theorem}

Note that if a graph has a weakly bi-complete $r$-sequence of order $t$, then, by arbitrarily adding additional vertices of the graph to the $r$-sets to obtain $r'$-sets, the graph also has a weakly bi-complete $r'$-sequence of order $t$ for any $r'$ satisfying $r \leq r' \leq n/2t$. This is useful for the third bound and for interpolating between the bounds. In particular, for the third bound, it will suffice to prove it for the case $r= 4\sqrt{p^{-1}\log n}$.

All three bounds on $t$ will follow from reducing the problem to a special case of the Zarankiewicz problem in which we want to guarantee a $K_{t,t}$ in a bipartite graph with parts of order at least $pn/16r$ and edge density at least $1-e^{-pr^2/8}$. Although it is not difficult to improve our bounds by being a little more careful at a few points in the argument, we have chosen to present proofs which determine the correct behavior while remaining as simple as possible.

By considering a random graph with edge density $p$, we see that the bounds in Theorem \ref{thm1} are close to being tight. Noting that every graph or its complement has edge density at least $1/2$, we have the following immediate corollary of the second bound. 

\begin{corollary} \label{cor1} For $r$ fixed and $n$ sufficiently large, 
$$g(r,n) \geq  \frac{1}{16}e^{2^{-6}r^2}\log n.$$
\end{corollary}

{\it Clique minors} are a strengthening of weakly complete sequences, with the added constraint that the sets $S_i$ are required to be connected. A classical result of Mader~\cite{Ma67} guarantees that for each $t$ there is $c(t)$ such that every graph on $n$ vertices with at least $c(t)n$ edges contains a clique minor of order $t$. Kostochka~\cite{Ko82,Ko84} and Thomason~\cite{Th84} independently determined the order of $c(t)$, proving that $c(t)=\Theta(t\sqrt{\log t})$. Almost two decades later, Thomason~\cite{Th01} determined an asymptotic formula: $c(t)=(\alpha+o(1))t \sqrt{\ln t}$, where $\alpha = 0.319...$ is a computable constant. 

In recent years, there has been a push towards extending these classical results on the extremal problem for graph minors to small  graph minors, that is, where few vertices are used in making the minor. A result of Fiorini, Joret, Theis and Wood~\cite{FJRW} says that, for each $t$, there are $h(t)$ and $f(t)$ such that every graph with at least $f(t)n$ edges contains a $K_t$-minor with at most $h(t)\log n$ vertices. The $\log n$ factor here is necessary. Indeed, for each $C$ there is $c>0$ and an $n$-vertex graph with at least $Cn$ edges and girth (which is defined as the length of the shortest cycle, but is also the minimum number of vertices in a $K_3$-minor) at least $c\log n$. Fiorini et al.~also conjectured that, for each $\epsilon>0$, one may take $f(t)=c(t)+\epsilon$ and $h(t)=C(\epsilon,t)$. Shapira and Sudakov~\cite{ShSu} came close to proving this conjecture, showing that  every $n$-vertex graph with at least $(c(t)+\epsilon)n$ edges contains a $K_t$-minor of order at most  $C(\epsilon,t)\log n \log \log n$. Building upon their approach, Montgomery~\cite{Mont} then solved the conjecture by removing the $\log \log n$ factor. 

These results are all about finding clique minors in sparse graphs. Here, we study the dense case and find conditions on $t$ and $r$ such that every dense graph contains a $K_t$-minor where each connected set corresponding to a vertex of the minor contains at most $r$ vertices. In particular, this shows that we may prove an analogue of Corollary~\ref{cor1} where the required subgraph is a clique minor rather than just a weakly complete $r$-sequence.

\begin{theorem}\label{minortheorem}
Let $n$ be sufficiently large and $G$ be a graph with $n$ vertices and edge density $p \geq n^{-1/8}$. 
If $24p^{-1/2} \leq r \leq \frac{1}{2}\sqrt{p^{-1}\log n}$ and $t \leq \frac{1}{32}e^{pr^2/256}\log n$, then $G$ contains a $K_t$-minor such that 
the connected sets corresponding to its vertices have size at most $8r$. 
\end{theorem}

Minors in which the connected sets corresponding to vertices have small diameter are known as {\it shallow minors}. This concept was introduced in a paper by Plotkin, Rao and Smith \cite{PRS}, though they attribute the idea to Leiserson and Toledo. Shallow minors also play a fundamental role in the work of Ne\v set\v ril and Ossona de Mendez on the theory of nowhere dense graphs (see their book \cite{NM}). 

We mention this concept because the proof of Theorem~\ref{minortheorem} also gives that the connected subset corresponding to each vertex has diameter at most $9$. A variant of this argument (using a different version of dependent random choice) can be used to reduce the diameter of the sets to $3$, but with a slightly weaker bound on $t$. We also note that there are analogues of Theorem \ref{thm1} when $r$ is larger or smaller than the assumed range. However, the proof is the same, so we omit the details.

We begin by proving Theorem \ref{thm1} and then deduce Theorem \ref{minortheorem}. We will make use of the following three lemmas. 

\begin{lemma}\label{firsta1}
Let $H=(V_1,V_2,E)$ be a bipartite graph with edge density $p$. There is a subset $B \subset V_2$ with $|B| \geq p |V_2|/2$ such that every vertex in $B$ has more than $p|V_1|/2$ neighbors in $V_1$. 
\end{lemma}

\begin{proof}
Delete all vertices in $V_2$ of degree at most $p|V_1|/2$ and let $B$ be the remaining subset of $V_2$. The number of deleted edges is at most $p|V_1||V_2|/2$ and hence there are at least $p|V_1||V_2|-p|V_1||V_2|/2=p |V_1||V_2|/2$ remaining edges from $B$ to $V_1$. As each vertex in $B$ is in at most $|V_1|$ edges, $|B| \geq p |V_1||V_2|/2|V_1|=p |V_2|/2$. 
\end{proof} 

\begin{lemma}\label{firstbd}
Let $H=(V_1,V_2,E)$ be a bipartite graph with edge density $1-q$. Then there is a subset $B \subset V_2$ with $|B| \geq |V_2|/2$ such that every vertex in $B$ has more than $(1-2q)|V_1|$ neighbors in $V_1$. 
\end{lemma}

\begin{proof}
Delete all vertices in $V_2$ of degree at most $(1-2q)|V_1|$ and let $B$ be the remaining subset of $V_2$. The number of nonedges touching the deleted vertices is at least $2q |V_1||V_2 \setminus B|$ and at most $q|V_1||V_2|$. Hence, $|V_2 \setminus B| \leq |V_2|/2$ and $|B| \geq |V_2|/2$.
\end{proof} 

\begin{lemma}\label{secondb}
If $H=(V_1,V_2,E)$ is a bipartite graph in which every vertex in $V_2$ has at least $p|V_1|$ neighbors in $V_1$, then there is a partition $V_1=A_1 \cup \ldots \cup A_{d}$ into subsets of order $r$ (so $d=|V_1|/r$) such that the fraction of pairs $(A_i,b)$ with $b \in V_2$ for which $b$ does not have a neighbor in $A_i$ is at most $(1-p)^r \leq e^{-pr}$.  
\end{lemma}

\begin{proof}
Partition $V_1$ uniformly at random into subsets $A_i$ of size $r$. The probability that $b$ has no neighbor in a  subset $A_i$ chosen uniformly at random is precisely the same as $A_i$ not containing any of the at least $p|V_1|$ neighbors of $b$ in $V_1$, which is at most ${(1-p)|A| \choose r}/{|A| \choose r} \leq (1-p)^r$. Therefore, the expected fraction of pairs $(A_i,b)$  for which $b$ does not have a neighbor in $V_1$ is at most $(1-p)^r$. Hence, there is such a partition of $V_1$  where the fraction of pairs $(A_i,b)$ is at most this expected value. 
\end{proof} 

\vspace{3mm}
\noindent {\bf Proof of Theorem \ref{thm1}:} Let $G=(V,E)$ be a graph on $n$ vertices with edge density $p$. By considering a random equitable vertex partition of $G$, there is a vertex partition $V=V_1 \cup V_2$ into parts of order $n/2$ such that the bipartite graph induced by this partition has edge density at least $p$. By Lemma \ref{firsta1}, there is $B \subset V_2$ with $|B| \geq  p|V_2|/2 \geq p n/4$ such that every vertex in $B$ has at least $p|V_1|/2$ neighbors in $V_1$. By Lemma \ref{secondb}, there is a partition  $V_1=A_1 \cup \ldots \cup A_{d}$ into subsets of order $r$ (so $d=|V_1|/r=n/2r$) such that the fraction of pairs $(A_i,b)$ with $b \in B$ for which $b$ does not have a neighbor in $V_1$ is at most $\rho:=(1-p/2)^r$. 

Consider the auxiliary bipartite graph $X$ with parts $\{1,\ldots,d\}$ and $B$, where $i$ is adjacent to $b \in B$ if there is at least one edge from $b$ to $A_i$. The density of $X$ between its parts is at least $1-\rho$. 

\vspace{0.1cm}
\noindent {\it Case 1}: $p \leq 3/r$. In this case, we have $\rho \leq e^{-pr/2} \leq 1-pr/4$ and hence the density of $X$ between its parts is at least $1-\rho \geq pr/4$.  

Let $S \subset \{1,\ldots,d\}$ consist of those vertices with at least $pr|B|/8$ neighbors in $B$. By Lemma \ref{firsta1}, we have $|S| \geq prd/8 = pr(n/2r)/8=pn/16$.  By Lemma \ref{secondb}, there is a partition  $B=B_1 \cup \ldots \cup B_{h}$ into subsets of order $r$ (so $h=|B|/r \geq p n /4r$) such that the fraction of pairs $(i,j)$ with $i \in S$ and $j \in [h]$ for which $i$ does not have a neighbor in $B_j$ in $X$ (and hence $A_i$ does not have an edge to $B_j$ in $G$) is at most $(1-pr/8)^r \leq e^{-pr^2/8}$.  

\vspace{0.1cm}
\noindent {\it Case 2}: $p>3/r$. In this case, we have $\rho = (1-p/2)^r \leq e^{-pr/2}$ and hence the density of $X$ between its parts is at least $1-\rho \geq 1-e^{-pr/2}$. 

Let $S \subset \{1,\ldots,d\}$ consist of those vertices with at least $(1-2e^{-pr/2}) |B| \geq (1-e^{-pr/4})|B|$ neighbors in $B$. By Lemma \ref{firstbd}, we have $|S| \geq d/2 =n/4r$.  By Lemma \ref{secondb}, there is a partition  $B=B_1 \cup \ldots \cup B_{h}$ into subsets of order $r$ (so $h=|B|/r \geq p n /4r$) such that the fraction of pairs $(i,j)$ with $i \in S$ and $j \in [h]$ for which $i$ does not have a neighbor in $B_j$ in $X$ (and hence $A_i$ does not have an edge to $B_j$ in $G$) is at most $\left(e^{-pr/4}\right)^r =e^{-pr^2/4}$.  

\vspace{0.1cm}
In either case, we obtain a bipartite graph $T$ with parts $S$ and $[h]$ where $(i,j) \in S \times [h]$ is an edge if $A_i$ has at least one edge to $B_j$, the parts are of order at least $N := pn/16r$ and the edge density is $1 - \delta$ for some $\delta \leq e^{-pr^2/8}$. Note that any $K_{t,t}$ in $T$ forms a weakly bi-complete $r$-sequence of order $t$ in $G$. 

If $r \geq 4\sqrt{p^{-1}\log n}$, then $pr^2 \geq 16 \log n$ and this edge density is at least $1-n^{-2}$, so $T$ is a complete bipartite graph with parts of order at least $pn/16r$. This gives the third desired bound. 

A classical result of K\H{o}v\'ari, S\'os and Tur\'an \cite{KST} on the Zarankiewicz problem shows that if a bipartite graph $T$ with parts of order at least $N$ has density at least $1-\delta$ and $N{(1-\delta)N \choose t} > (t-1){N \choose t}$, then the bipartite graph contains $K_{t,t}$. Notice that this inequality holds if $(1-\delta-\frac{t}{N})^t \geq t/N$.

If $r \leq 2p^{-1/2}$, $p \geq n^{-1/3}$ and $t \leq \frac{\log n}{4\log (32/pr^2)}$, then, letting $x=pr^2/8 \leq 1/2$, we see that $T$ has edge density at least $1-e^{-x} \geq x/2$. However, $x \geq 4t/N$, so $T$ contains a $K_{t,t}$ if $(x/4)^t \geq t/N$. But
$$(x/4)^t \geq n^{-1/4} \geq 32p^{-3/2}tn^{-1} \geq  t16r/pn = t/N,$$
and we have shown the first desired bound. 

Suppose now that we are trying to obtain the second desired bound. Since $4p^{-1/2} \leq r < \sqrt{p^{-1}\log n}$, $p \geq n^{-1/5}$ and $t \leq \frac{1}{16}e^{pr^2/8}\log n$, we have $\delta \geq e^{-pr^2/8} \geq n^{-1/8}$ and $t \leq n^{1/4}$. Therefore, we have $\delta N = \delta pn/16r \geq pn^{7/8}/16r \geq n^{1/2} \geq t$ and $t/N \leq n^{-1/4}$, so the desired inequality for the K\H{o}v\'ari--S\'os--Tur\'an result holds if $(1-2\delta)^t \geq n^{-1/4}$. Taking the logarithm of both sides and noting that $\log(1-2\delta) \geq -4\delta$ as $\delta \leq e^{-2}$, we see that this holds as long as $t \leq \frac{\log n}{16\delta}\leq\frac{1}{16}e^{pr^2/8}\log n$, which is a given assumption. This completes the proof of Theorem \ref{thm1}. \qed

\vspace{3mm}
To prove Theorem~\ref{minortheorem}, we combine the previous embedding technique with the following consequence of dependent random choice (discussed in the previous section) taken from~\cite{FLS}.   

\begin{lemma}\label{drclemma}
Let $H=(U,V,E)$ be a bipartite graph with $|U|=|V|=n/2$ and at least $pn^2/4$ edges. Then, if $p^2 n \geq 1600$, there is a subset $X \subset U$ with $|X| \geq pn/50$ such that for every pair of vertices $x,y \in X$, there are at least $10^{-9}p^5n$ internally vertex-disjoint paths with four edges between $x$ and $y$ with internal vertices not in $X$. 
\end{lemma}

{\bf Proof of Theorem~\ref{minortheorem}:} Let $G$ be a graph with edge density $p$ on $n$ vertices, so it has $p{n \choose 2} \geq pn^2/4$ edges. By deleting vertices of degree less than $pn/8$ one at a time, we arrive at a subgraph $G'$ with minimum degree at least $pn/8$ and at least $pn^2/4-n(pn/8)=pn^2/8$ edges. Let $v$ denote the number of vertices in $G'$, so $p^{1/2}n/4 \leq v \leq n$.

Let $H=(U,V,E)$ be a bipartite subgraph of $G'$ with parts of order $v/2$ and at least $pn^2/32$ edges such that the minimum degree of $H$ is at least $pn/32$. Such a bipartite subgraph exists by considering a random bipartition. Note that the number of edges in $H$ is at least $pn^2/32=(pn^2/8v^2)v^2/4$. By Lemma \ref{drclemma}, there is a subset $X \subset U$ with $|X| \geq (pn^2/8v^2)v/50 \geq pn/400$ such that for every pair of vertices $x,y \in X$, there are at least $10^{-9}(pn^2/8v^2)^5v \geq 10^{-14}p^5n$ internally vertex-disjoint paths with four edges between $x$ and $y$ with internal vertices not in $X$. Let $X'$ be an arbitrary subset of $X$ of size exactly $pn/400$.  

As every vertex in $X$ (and hence $X'$) has degree at least $pn/32$ in $H$, there are at least $(pn/32)|X'|$ edges between $X'$ and $V$. Delete all vertices in $V$ with fewer than $(pn/32v)|X'|$ neighbors in $X'$ and let $Z$ be the remaining subset of $V$. The number of edges between $X'$ and $Z$ is at least $(pn/32)|X'|-(pn/32v)|X'|(v/2)=pn|X'|/64$. Note that $|Z| \geq pn/64>|X'|$. Let $Z' \subset Z$ be a subset with $|Z'|=|X'|$ such that the number of edges between $X'$ and $Z'$ at least 
$p|X'||Z'|/64$. Such a subset $Z'$ exists by considering a random subset of $Z$ of order $|Z'|$. 
Consider the bipartite graph $H'$ between $X'$ and $Z'$. It has $2|Z'|$ vertices and at least $p|X'||Z'|/64=(p|X'|/64|Z'|)(2|Z'|)^2/4$ edges. Applying Lemma \ref{drclemma} to $H'$, there is a subset $Y \subset Z'$ with $|Y| \geq (p|X'|/64|Z'|)(2|Z'|)/50=p|X'|/1600$ such that for every pair of vertices $x,y \in Y$, there are at least $10^{-9} (p|X'|/64|Z'|)^5(2|Z'|) \geq 10^{-18}p^5 |X'|$ internally vertex-disjoint paths with four edges between $x$ and $y$ with internal vertices not in $Y$. Since every vertex in $Z$ has at least $(pn/32v)|X'|$ neighbors in $X'$, the density between $X'$ and $Y$ is at least $pn/32v$. Let $W \subset X'$ be a subset of order $|Y|$ such that the edge density between $W$ and $Y$ is at least $pn/32v \geq p/32$. Such a subset exists by considering a random subset of $X'$ of order $|Y|$. 

By Theorem \ref{thm1} (or rather its proof, as we pass to a balanced bipartite subgraph), the bipartite graph between $W$ and $Y$ contains a weakly bi-complete $r$-sequence of order $t$ with $t = \frac{1}{16}e^{pr^2/256}\log |Y| \geq \frac{1}{32}e^{pr^2/256}\log n$. For each part $A$ used to make this weakly bi-complete $r$-sequence, fix a vertex $a \in A$ and consider any other vertex $b \in A$. There are at least $10^{-18}p^5|X'|>10^{-21}p^6 n>8rt$ internally vertex disjoint paths, so we can find one of these internal paths so that the vertices have not already been used and add the three internal vertices of the path to connect $a$ and $b$. Doing this for every vertex $b \in A \setminus a$, we get a connected set $A'$ with at most $4r$ vertices. We can do this for each of the $2t$ sets $A$ making up the weakly bi-complete $r$-sequence of order $t$. We thus obtain a $K_{t,t}$-minor with each part corresponding to a vertex of order at most $4r$. From a matching in the $K_{t,t}$, we get a $K_t$-minor with each part corresponding to a vertex of order at most $8r$. \qed

To prove the claim that each set in the $K_t$-minor may be chosen to have diameter at most $9$, suppose that $A$ and $B$ are two sets in the weakly bi-complete $r$-sequence and the union of the sets $A'$ and $B'$ formed from $A$ and $B$ corresponds to a vertex in the $K_t$-minor. If we let $ab$ be an edge with $a \in A$ and $b \in B$, we see that we could have chosen the sets $A'$ and $B'$ so that every vertex in $A'$ is connected to $a$ by a path of length $4$ and every vertex in $B'$ is connected to $b$ by a path of length $4$. Since $a$ and $b$ are joined, this clearly implies that $A' \cup B'$ has diameter at most $9$.

\section{Induced Ramsey numbers and Ruzsa--Szemer\'edi graphs} \label{sec:matchings}

A graph $H$ is said to be an {\it induced subgraph} of $G$ if $V(H) \subset V(G)$ and two vertices of $H$ are adjacent if and only if they are adjacent in $G$. We write $G \rightarrow_{ind} (H_1,H_2)$ if every red/blue-coloring of the edges of $G$ contains an induced copy of $H_1$ all of whose edges are red or an induced copy of $H_2$ all of whose edges are blue. The {\it induced Ramsey number} $r_{\textrm{ind}}(H_1,H_2)$ is the minimum $N$ for which there exists a graph $G$ on $N$ vertices with $G \rightarrow_{ind} (H_1,H_2)$. When $H_1 = H_2 = H$, we simply write $r_{\textrm{ind}}(H)$.

We will be interested in the induced Ramsey number of trees. It is easy to show that the usual Ramsey number for trees is linear in the number of vertices. For some trees, such as paths, it is even possible to show~\cite{HKL95} that the induced Ramsey number grows linearly in the number of vertices. However, Fox and Sudakov~\cite{FS08} showed that there exist trees $T$ for which the induced Ramsey number grows superlinearly. More precisely, they showed that $r_{\textrm{ind}}(K_{1, t}, M_t)$ is superlinear in $t$, where $M_t$ is the matching with $t$ edges. It is then sufficient if $T$ contains both $K_{1,t}$ and $M_t$ as induced subgraphs.

Fox and Sudakov~\cite{FS08} proved their result by an appeal to the regularity lemma. In this section, we prove a strengthening of this result by showing that there is a close connection between $r_{\textrm{ind}}(K_{1, t}, M_n)$ and the celebrated Ruzsa--Szemer\'edi induced matching problem~\cite{RuSz}. We say that a graph $G=(V,E)$ is an {\it $(n,t)$-Ruzsa--Szemer\'edi graph} (or an $(n,t)$-RS graph, for short) if its edge set is the union of $t$ pairwise disjoint induced matchings, each of size $n$. 

\begin{theorem} \label{thm:toRS}
If $G \rightarrow_{\textrm{ind}} (K_{1,t},M_n)$, then $G$ contains a subgraph which is an $(n,t)$-RS graph. 
\end{theorem}

\begin{proof}
Pick out disjoint induced matchings of size $n$ from $G$ until there are no more induced matchings of this size. If at least $t$ induced matchings are picked out, then the union of these $t$ induced matchings makes the subgraph of $G$ which is an $(n,t)$-RS graph. Otherwise, we color the edges of the (fewer than $t$) induced matchings in red and the remaining edges in blue. The red graph won't contain a red $K_{1,t}$ as each of the (fewer than $t$) induced matchings contributes at most one to the degree of each vertex. Moreover, since we cannot pick out another disjoint induced matching, the blue graph does not contain an induced $M_n$, contradicting our assumption that $G \rightarrow_{\textrm{ind}} (K_{1,t},M_n)$.
\end{proof}

To recover the quantitative statement that $r_{\textrm{ind}}(K_{1, t}, M_t)$ is superlinear in $t$, we use Theorem~\ref{thm:toRS} to deduce that if $G \rightarrow_{ind} (K_{1,t},M_t)$ then $G$ contains a $(t,t)$-RS graph. However, the work of Ruzsa and Szemer\'edi~\cite{RuSz} shows that any such graph must have $t (\log^*t)^c$ vertices for some positive constant $c$, where $\log^*t$ is the slowly-growing function defined by $\log^* t = 0$ if $t \leq 1$ and $\log^*t = 1 + \log^*(\log t)$ if $t > 1$. Fox's bound for the removal lemma~\cite{Fo}, which we will discuss further in the next section, improves this estimate to $t e^{c \log^*t}$ for some positive constant $c$. This yields the following corollary.

\begin{corollary}
There exists a positive constant $c$ such that $r_{\textrm{ind}}(K_{1, t}, M_t) \geq t e^{c \log^* t}$.
\end{corollary}

To show that Ruzsa--Szemer\'edi graphs also give rise to graphs $G$ for which $G \rightarrow_{\textrm{ind}} (K_{1,t},M_n)$, we first show that we may assume our Ruzsa--Szemer\'edi graph is bipartite.

\begin{lemma} \label{lem:tobip}
If there is an $(n,t)$-RS graph $G=(V,E)$ on $N$ vertices, then there is a bipartite $(2n,t)$-RS graph $B$ on $2N$ vertices.
\end{lemma}

\begin{proof}
Let $B$ be the bipartite graph with parts $V_1$ and $V_2$, each a copy of $V$, where $(u,v) \in V_1 \times V_2$ is an edge if and only if $(u,v)$ is an edge of $G$. Each of the $t$ induced matchings of size $n$ in $G$ corresponds to an induced matching of size $2n$ in $B$ and these induced matchings  make up $B$. 
\end{proof}

It is now straightforward to show that bipartite Ruzsa--Szemer\'edi graphs $G$ satisfy $G \rightarrow_{\textrm{ind}} (K_{1,n},M_n)$.

\begin{theorem} \label{thm:toind}
Suppose that $c \geq 2$ and $G$ is a bipartite $(cn,N/c)$-RS graph on $N$ vertices. Then $G \rightarrow_{ind} (K_{1,n},M_n)$. 
\end{theorem}

\begin{proof}
In any red/blue-coloring of the edges of $G$, at least half of the edges are blue or half of them are red. In the former case, at least one of the induced matchings of size $cn$ has at least half of its edges in color blue and, since $cn/2 \geq n$, these edges form a blue induced matching of size $n$. In the latter case, as there are $nN$ edges and so at least $nN/2$ red edges, there is a vertex of red degree at least $n$. Since $G$ is bipartite, this induces a red star $K_{1,n}$. 
\end{proof}

As observed by Ruzsa and Szemer\'edi~\cite{RuSz}, a construction due to Behrend~\cite{Be46} allows one to show that there are $(N/e^{c \sqrt{\log N}}, N/5)$-RS graphs on $N$ vertices. Applying Lemma~\ref{lem:tobip} and Theorem~\ref{thm:toind}, we get the following corollary. 

\begin{corollary}
There exists a constant $c$ such that $r_{\textrm{ind}}(K_{1, t}, M_t) \leq t e^{c \sqrt{\log t}}$.
\end{corollary}



\section{Colored triangle removal} \label{sec:coloredremoval}

The triangle removal lemma of Ruzsa and Szemer\'edi \cite{RuSz} states that for each $\epsilon>0$ there is $\delta>0$ such that every graph on $n$ vertices with at most $\delta n^3$ triangles can be made triangle-free by removing at most $\epsilon n^2$ edges. That is, a graph with a subcubic number of triangles can be made triangle-free by removing a subquadratic number of edges. The triangle removal lemma has many applications in graph theory, additive combinatorics, discrete geometry and theoretical computer science.

Until recently, the only known proof of the triangle removal lemma used Szemer\'edi's regularity lemma~\cite{Sz76} and gave a weak quantitative bound for $\delta^{-1}$, namely, a tower of $2$s of height polynomial in $\epsilon^{-1}$. Recently, Fox \cite{Fo} found a new proof which avoids the regularity lemma and improves the bound on $\delta^{-1}$ to a tower of $2$s of height logarithmic in $\epsilon^{-1}$. It remains a major open problem to find a bound of constant tower height.

It is easy to show that the triangle removal lemma is equivalent to the following statement: for each $\epsilon>0$ there is $\delta > 0$ such that every tripartite graph on $n$ vertices whose edge set can be partitioned into $\epsilon n^2$ triangles contains at least $\delta n^3$ triangles. We present a simple proof of the following Ramsey-type weakening of this statement with a much better (only double-exponential) bound.

\begin{theorem}\label{main}
Every $r$-edge-coloring of a tripartite graph $G$ with parts $V_0,V_1,V_2$, where $|V_1|=|V_2|=n$, $|V_0| = cn$ and $V_1$ is complete to $V_2$, such that there is a collection of $n^2$ edge-disjoint monochromatic triangles which cover the edges between $V_1$ and $V_2$ contains at least $(4cr)^{-2^{r+3}}n^3$ monochromatic triangles.
\end{theorem}

When $(4cr)^{-2^{r+3}}n^3>n^2$, two of the monochromatic triangles must share an edge between $V_1$ and $V_2$ and so we have the following corollary.

\begin{corollary} \label{cor:col}
Suppose $n>(4cr)^{2^{r+3}}$. Every $r$-edge-coloring of a tripartite graph $G$ with parts $V_0,V_1,V_2$, where $|V_1|=|V_2|=n$, $|V_0| = cn$ and $V_1$ is complete to $V_2$, such that there is a collection of $n^2$ edge-disjoint monochromatic triangles which cover the edges between $V_1$ and $V_2$ contains a monochromatic diamond, i.e., an edge in two monochromatic triangles.
\end{corollary}

One application of the triangle removal lemma, noted by Solymosi \cite{So}, is a short proof for the corners theorem of Ajtai and Szemer\'edi \cite{AjSz}. This theorem states that for each $\epsilon>0$ there is $N(\epsilon)$ such that, for $N \geq N(\epsilon)$, any subset $S$ of the $N \times N$ grid with $|S| \geq \epsilon N^2$ contains a corner, i.e., the vertices $(x,y)$, $(x+d,y)$, $(x,y+d)$ of an isosceles right triangle. This in turn gives a simple proof of Roth's theorem \cite{Ro} that every subset of the integers of positive upper density contains a three-term arithmetic progression.

The best known upper bound for the corners theorem, due to Shkredov \cite{Sh}, states that any subset of the $N \times N$ grid with no corner has cardinality at most $N^2/(\log \log N)^{c}$, where $c>0$ is an absolute constant. Graham and Solymosi \cite{GrSo} proved a better bound for the Ramsey-type analogue of the corners theorem. They showed that there is $c>0$ such that any coloring of the $N \times N$ grid with fewer than $c\log \log N$ colors contains a monochromatic corner. This in turn implies a double-exponential bound for the van der Waerden number $W(3; r)$ (the smallest $N$ such that any $r$-coloring of the set $\{1, 2, \dots, N\}$ contains a monochromatic three-term arithmetic progression), but this is weaker than the exponential bound that follows from the best quantitative estimate for Roth's theorem~\cite{B15, San11}.

Just as the triangle removal lemma implies the corners theorem, Corollary \ref{cor:col} implies the Graham--Solymosi bound on monochromatic corners in colorings of the grid. Indeed, consider an $r$-coloring of the $N \times N$ grid with $N>(8r)^{2^{r+3}}$. Let $V_0$ denote the set of $2N-1$ lines with slope~$-1$ that each contain at least one of the grid points, $V_1$ denote the set of $N$ vertical lines that each contain $N$ of the grid points and $V_2$ denote the set of $N$ horizontal lines that each contain $N$ of the grid points. Consider the tripartite graph $G$ with parts $V_0,V_1,V_2$, where two lines are adjacent if and only if they intersect in one of the points of the $N \times N$ grid, and color the edge between them the color of their intersection point. Note that every line in $V_1$ intersects every line in $V_2$ in the grid, so $V_1$ is complete to $V_2$ in the graph. Moreover, the three lines passing through any grid point form a monochromatic triangle and this collection of $N^2$ monochromatic triangles gives an edge-partition of $G$. Therefore, by Corollary~\ref{cor:col}, $G$ contains a monochromatic diamond. This, in turn, implies that the coloring of the grid must contain a monochromatic corner. 

The proof of Theorem \ref{main} follows from iterating the following simple lemma.

\begin{lemma}\label{mainlemma}
In every $r$-edge coloring of a tripartite graph $G$ with parts $V_0,V_1,V_2$, where $|V_1|=|V_2|=n$, $|V_0|=cn$ and the number $m$ of edges between $V_1$ and $V_2$ is at least $n^2/2$, such that there are fewer than $\delta n^3$ monochromatic triangles and there is a collection of edge-disjoint monochromatic triangles
which cover the edges between $V_1$ and $V_2$, there are subsets $V_1' \subset V_1$ and $V_2' \subset V_2$
with $|V_1'|=|V_2'| \geq \frac{n}{4cr}$ and a color such that the number of edges of that color between $V_1'$ and $V_2'$ is at most $4\delta n^2$.
\end{lemma}

\begin{proof}
Let $\mathcal{T}$ denote a collection of edge-disjoint monochromatic triangles which cover the edges of $G$ between $V_1$ and $V_2$, so $|\mathcal{T}|=m$. Partition $V_0=A \cup B$, where $v \in V_0$ is in $A$ if the number of triangles in $\mathcal{T}$ containing $v$ is at least $\frac{m}{2|V_0|}$. The number of triangles in $\mathcal{T}$ containing a vertex in $B$ is less than $|B|\frac{m}{2|V_0|} \leq \frac{m}{2}$. Hence, at least
half of the triangles in $\mathcal{T}$ contain a vertex from $A$.

Since each vertex in $A$ is in at most $n$ triangles from $\mathcal{T}$, we have $|A| \geq (m/2)/n \geq n/4$. If every vertex in $A$ is in at least $4\delta n^2$ monochromatic triangles, the total number of monochromatic triangles is at least $|A| 4\delta n^2 \geq \delta n^3$. Hence, there is a
 vertex $v \in A$ in fewer than $4\delta n^2$ monochromatic triangles. Since $v \in A$, there is a color, say red, such that $v$ is in at least
$\frac{m}{2|V_0|r} \geq \frac{n^2}{4cnr} = \frac{n}{4cr}$ monochromatic red triangles from $\mathcal{T}$. For $i=1,2$, let $V_i'$ denote those vertices in $V_i$ which are in a monochromatic red triangle from $\mathcal{T}$ with vertex $v$, so $|V_1'|=|V_2'| \geq \frac{n}{4cr}$. Since $v$ is in fewer than $4\delta n^2$ monochromatic triangles and $v$ is complete in red to $V_1'$ and $V_2'$, the number of red edges between $V_1'$ and $V_2'$ is at most $4\delta n^2$, which completes the proof.
\end{proof}

\vspace{3mm}

{\bf Proof of Theorem \ref{main}:} Let $f(n,r,q,s)$ be the minimum number of monochromatic triangles in an $r$-edge coloring of a tripartite graph $G$ with parts $V_0,V_1,V_2$, where $|V_1|=|V_2| \geq n$ and $|V_0| \leq q$, such that the number of edges between $V_1$ and $V_2$ is at least $|V_1|^2-s$ and the edges between $V_1$ and $V_2$ can be covered by edge-disjoint monochromatic triangles. In Theorem \ref{main}, $V_1$ is complete to $V_2$ and $|V_0|=cn$, so we are trying to prove a lower bound on $f(n,r,cn,0)$, namely $$f(n,r,cn,0) \geq (4cr)^{-2^{r+3}}n^3.$$

When $r=1$, Lemma \ref{mainlemma} implies that if $G$ is a tripartite graph $G$ with parts $V_0,V_1,V_2$, where $|V_1|=|V_2|=n$, $|V_0|=cn$ and $e(V_1,V_2) \geq n^2/2$, such that $G$ contains at most $\delta n^3$ triangles and there is a collection of edge-disjoint triangles in $G$ which cover the edges between $V_1$ and $V_2$, then there are subsets $V_1' \subset V_1$ and $V_2' \subset V_2$ with $|V_1'|=|V_2'| \geq \frac{n}{4c}$ for which $e(V_1',V_2') \leq 4\delta n^2$. In particular, if $e(V_1,V_2)=n^2-s$, then $$4 \delta n^2 \geq e(V_1',V_2') \geq |V_1'||V_2'|-s \geq \left(\frac{n}{4c}\right)^2-s,$$ and 
hence $\delta n^3 \geq \frac{n^3}{64c^2}-\frac{ns}{4}$. Substituting $q=cn$, we get the bound 
$$f(n,1,q,s) \geq  \frac{n^5}{64q^2}-\frac{ns}{4}.$$

When $s<n^2/2$, Lemma~\ref{mainlemma} implies that by deleting the edges of the sparsest color between $V_1'$ and $V_2'$ and letting $f_0=f(n,r,q,s)$, we have
$$f_0> f\left(\frac{n^2}{4qr},r-1,q,s+4f_0/n\right).$$ 
Let $n_i=\frac{n^{2^i}}{(4qr)^{2^i-1}}$, $s_0=s$ and, for $i\geq 1$, $s_i=s_{i-1}+4f_0/n_{i-1}$, so $s_i<s+4f_0\sum_{j=0}^{i-1}\frac{1}{n_{j}} <s+5f_0/n_{i-1}$. After $i$ iterations, if $s_{i-1}<n_{i-1}^2/2$, we get
$$f_0> f\left(n_i,r-i,q,s_i\right).$$

We set $s=0$ and $q=cn$. We use the above inequalities to compute a lower bound on $f_0=f(n,r,cn,0)$.  
Either we have $s_{i-1} \geq n_{i-1}^2/2$ for some $i<r$ or $f_0 \geq f\left(n_{r-1},1,cn,s_{r-1}\right)$. 
In the first case, we have $5f_0/n_{i-2} \geq s_{i-1} \geq n_{i-1}^2/2$ and, therefore,
$$f_0 \geq \frac{n_{i-2}n_{i-1}^2}{10} \geq \frac{n_{r-3}n_{r-2}^2}{10} \geq (4cr)^{-2^r}n^3.$$   
In the second case, we have 
$$f_0 \geq f\left(n_{r-1},1,cn,s_{r-1}\right) \geq \frac{n_{r-1}^5}{64q^2}-\frac{ns_{r-1}}{4} \geq \frac{n^3}{(4cr)^{5(2^{r-1}-1)}64c^2}-\frac{5}{4}f_0\frac{n}{n_{r-2}},$$ 
in which case we get $$f_0 \geq \frac{1}{2}\frac{n_{r-2}}{n} \frac{n^3}{(4cr)^{5(2^{r-1}-1)}64c^2} \geq (4cr)^{-2^{r+3}}n^3.$$ 
In either case, we have the desired inequality, which completes the proof of Theorem \ref{main}. \qed

\vspace{3mm}
\noindent
{\bf Acknowledgements.} We would like to thank the anonymous referees for their helpful remarks and Zoltan F\"uredi for bringing the reference~\cite{F91} to our attention.

\end{document}